\documentclass[11pt,a4paper]{article}

\usepackage{amsmath}
\usepackage{amsthm}
\usepackage{amssymb}
\usepackage{cain_arxiv}
\usepackage{graphicx}

\theoremstyle{definition}
\newtheorem{definition}{Definition}[section]

\newtheorem{example}[definition]{Example}
\newtheorem{problem}[definition]{Open Problem}

\theoremstyle{plain}
\newtheorem{proposition}[definition]{Proposition}
\newtheorem{lemma}[definition]{Lemma}
\newtheorem{theorem}[definition]{Theorem}
\newtheorem{corollary}[definition]{Corollary}

\def\defterm#1{\textit{#1}}

\def\nset{\mathbb{N}}

\def\adjone#1{{#1}^{{\tt 1}}}

\newcommand{\pres}[3][]{#1\langle #2\:#1|\allowbreak\:#3 #1\rangle}
\newcommand{\gpres}[3][]{\mathrm{Gp}#1\langle #2\:#1|\allowbreak\:#3 #1\rangle}
\newcommand{\sgpres}[3][]{\mathrm{Sg}#1\langle #2\:#1|\allowbreak\:#3 #1\rangle}

\newcommand{\sgmpres}[3][]{\mathrm{SgM}#1\langle #2\:#1|\allowbreak\:#3 #1\rangle}

\def\BSG{\mathrm{G}}
\def\BSS{\mathrm{S}}

\newcommand{\malcgen}[1]{#1^{{\rm M}}}

\newcommand{\rev}{\mathrm{rev}}

\DeclareMathOperator{\cotan}{cotan}

\def\elt#1{\overline{#1}}

\def\path#1{\widehat{#1}}

\def\rel#1{\mathcal{#1}}

\let\emptyword=\varepsilon

\def\fsa#1{\mathcal{#1}}

\begin{document}

\title{Automatic structures for subsemigroups of Baumslag--Solitar semigroups} 
\author{Alan J. Cain}
\date{}

\maketitle

\address{%
School of Mathematics \& Statistics, University of St Andrews\\ North
Haugh, St Andrews, Fife KY16 9SS, United
Kingdom\footnote{\textit{Current address:} Centro de Matem\'{a}tica,
  Universidade do Porto, Rua do Campo Alegre 687, 4169--007 Porto,
  Portugal.}
}
\email{%
ajcain@fc.up.pt
}
\webpage{%
www.fc.up.pt/pessoas/ajcain/
}


\begin{abstract}
This paper studies automatic structures for subsemigroups of
Baumslag--Solitar semigroups (that is, semigroups presented by
$\pres{x,y}{(yx^m,x^ny)}$ where $m,n \in \nset$). A geometric argument
(a rarity in the field of automatic semigroups) is used to show that
if $m > n$, all of the finitely generated subsemigroups of this
semigroup are [right-] automatic. If $m < n$, all of its finitely
generated subsemigroups are left-au\-to\-mat\-ic. If $m=n$, there exist
finitely generated subsemigroups that are not automatic. An appendix
discusses the implications of these results for the theory of Malcev
presentations. (A Malcev presentation is a special type of
presentation for semigroups embeddable into groups.)
\end{abstract}

\section{Introduction}

Baumslag--Solitar groups, which have presentations of the form
\begin{equation}
\label{eq:bsg1}
\gpres{x,y}{(yx^m,x^ny)},
\end{equation}
where $m,n \in \nset$, were introduced to answer certain questions on
Hopficity \cite{ar:baumslag1962a}. They have proved to be of
independent interest, as have the semigroups with the same
presentation
\begin{equation}
\label{eq:bss1}
\sgpres{x,y}{(yx^m,x^ny)}.
\end{equation}
Denote the Baumslag--Solitar group \eqref{eq:bsg1} by $\BSG(m,n)$ and
the Baumslag--Solitar semigroup \eqref{eq:bss1} by $\BSS(m,n)$. As a
consequence of Adyan's Theorem \cite{adjan_embeddability}, the semigroup
$\BSS(m,n)$ embeds into the group $\BSG(m,n)$ for any $m,n \in \nset$.

The present paper studies automatic structures for finitely generated
subsemigroups of Baumslag--Solitar semigroups and the results obtained
confirm Jackson's \cite{ar:jackson2002a} observation of the radical
differences between Baumslag--Solitar groups and semigroups. The
concept of an automatic structure has been extended from groups
\cite{epstein_wordproc} to semigroups \cite{campbell_autsg}. While
automatic structures for groups have an elegant geometric
characterization known as the `fellow traveller' property, this
characterization does not extend to automatic structures for general
semigroups. Although two different geometric characterizations have
been established for automatic structures for general semigroups
\cite{hoffmann_geometric} and for a broad class that includes
the right-cancellative semigroups \cite{silva_geometric}, these are neither as
elegant nor as easy to use as the fellow traveller property. Thus, while
geometric reasoning plays a vital role in the theory of automatic
groups, the theory of automatic semigroups exhibits a paucity of such
arguments, tending instead towards syntactic analyses.

However, automatic structures for group-em\-bed\-dable semigroups do
inherit many of the pleasant geometric properties of automatic groups
(see \cite[Section~3]{crr_ssgofg} or
\cite[Section~2.3]{c_phdthesis}). Thus --- for example --- relatively
basic geometric arguments have been used successfully by Robertson,
Ru\v{s}kuc, and the present author to prove that all finitely
generated subsemigroups of virtually free groups are automatic
\cite[Theorem~4]{crr_ssgofg}. In the present paper, rather more
sophisticated geometric reasoning is deployed to prove that all
finitely generated subsemigroups of the Baumslag--Solitar semigroup
$\BSS(m,n)$ are automatic when $m > n$
(Theorem~\ref{thm:one}).

The result that finitely generated subsemigroups of $\BSS(m,n)$ are
left-au\-to\-mat\-ic when $m < n$ (Theorem~\ref{thm:two}) follows as a
corollary. (See Definition~\ref{def:autstruct} for the meaning of
left-au\-to\-mat\-ic.)

Hoffmann \cite[Corollary~4.20]{hoffmann_phd} used syntactic
reasoning to show that the Baumslag--Solitar semigroup $\BSS(m,n)$ is
automatic if $m > n$ and left-au\-to\-mat\-ic if $m < n$; these results follow as special cases of Theorems~\ref{thm:one} and~\ref{thm:two}.

However, as is shown in Section~\ref{sec:bssmmcontnonaut}, the
Baumslag--Solitar semigroup $\BSS(m,m)$ contains finitely generated
subsemigroups that are neither right- nor left-au\-to\-mat\-ic. This is in
some sense surprising --- one would expect that the structure of a
Baumslag--Solitar semigroup $\BSS(m,n)$ is more complex when $m \neq
n$ and simpler when $m = n$. Indeed, the Baumslag--Solitar group
$\BSG(m,n)$ is automatic if and only if $m=n$
\cite[Example~7.4.1]{epstein_wordproc}.

Malcev presentations are a special type of semigroup presentation for
semigroups embeddable into groups. Every automatic or left-au\-to\-mat\-ic
group-em\-bed\-dable semigroup admits a finite Malcev
presentation~\cite[Theorem~2]{crr_ssgofg}. The results of this paper
therefore have implications for the theory of Malcev presentations. An
appendix in Section~\ref{sec:malcev} discusses these implications and
poses some questions for further research.

[This paper is based on Chapter~8 of the author's Ph.D. thesis
  \cite{c_phdthesis}.]

\section{Preliminaries}

\subsection{Words, prefixes, and suffixes}
Following \cite{epstein_wordproc}, the notation used in this paper distinguishes a word from the element of the semigroup or group it represents. Let $A$ be an alphabet representing a set of generators for a semigroup $S$. For any word $w \in A^+$, denote by $\elt{w}$ the element of $S$ represented by $w$. Similarly, if $A$ represents a set of generators for a group $G$, let $\elt{w}$ be the element of $G$ represented by $w \in (A \cup A^{-1})^*$. In both cases, for any set of words $W$, $\elt{W}$ is the set of all elements represented by at least one word in $W$.

Denote the identity of $A^*$ --- the empty word --- by $\emptyword$. Denote the length of $u \in A^*$ by $|u|$.

Let $u = u_1\cdots u_n$, where $u_i \in A$. For $t \in \nset \cup \{0\}$, let
\[
u(t) = \begin{cases}
\emptyword &\text{if }t = 0,\\
u_1\cdots u_t &\text{if }0 < t \leq n,\\
u_1\cdots u_n &\text{if }n < t,
\end{cases}
\]
and let
\[
u[t] = \begin{cases}
u_{t+1}\cdots u_n &\text{if }0 \leq t < n,\\
\emptyword &\text{if }n \leq t.
\end{cases}
\]
So $u(t)$ is the prefix of $u$ up to and including the $t$-th letter; $u[t]$ is the suffix of $u$ after and not including the $t$-th letter. Observe that for all $t \in \nset \cup \{0\}$, $u = u(t)u[t]$, and that if one formally assumes that $u_t = \emptyword$ for $t > n$, then $u(t+1) = u(t)u_{t+1}$ and $u[t] = u_{t+1}u[t+1]$.

\subsection{Cayley graphs}
\begin{definition}
Let $S$ be a semigroup and $A$ an alphabet representing a set of generators for $S$. The \defterm{Cayley graph\/} $\Gamma(S,A)$ of $S$ with respect to $A$ is the directed graph with vertex set $S$ and, for each pair of vertices $s,t$, an edge from $s$ to $t$ labelled by $a \in A$ if and only if $s\elt{a} = t$.
\end{definition}

Let $S$ be a semigroup that embeds in a group. If $S$ is not already a monoid, then a new vertex representing a two-sided identity can be adjoined to $\Gamma(S,A)$. Throughout this paper, assume that $\Gamma(S,A)$ has a vertex representing a two-sided identity as a basepoint.

Given a word $u \in A^+$, let $\path{u}$ be the walk in $\Gamma(S,A)$ starting at the basepoint and labelled by $u$. The walk $\path{u}$ ends at the vertex $\elt{u}$.

\begin{definition}
Let $s,t$ be two vertices in $\Gamma(S,A)$. The [undirected] distance
from $s$ to $t$, denoted $d(s,t)$, is defined to be the infimum of the
lengths of the undirected paths from $s$ to $t$. [The assumption of
  the presence of a basepoint has the side-effect of ensuring at least
  one such undirected path exists.] 
\end{definition}

\subsection{Automatic semigroups}

This section contains the definitions and results for automatic
semigroups that are required in the remainder of the paper.

\begin{definition}
Let $\$$ be a new symbol not in $A$. Let $A(2,\$) = \{(a,b) : a,b \in A \cup \{\$\}\} - \{(\$,\$)\}$ be a new alphabet. Define the mapping $\delta_A : A^+ \times A^+ \to A(2,\$)^+$ by
\[
(u_1\cdots u_m,v_1\cdots v_n) \mapsto
\begin{cases}
(u_1,v_1)\cdots(u_m,v_n) & \text{if }m=n,\\
(u_1,v_1)\cdots(u_n,v_n)(u_{n+1},\$)\cdots(u_m,\$) & \text{if }m>n,\\
(u_1,v_1)\cdots(u_m,v_m)(\$,v_{m+1})\cdots(\$,v_n) & \text{if }m<n,
\end{cases}
\]
where $u_i,v_i \in A$. The symbol $\$$ is usually called the \defterm{padding symbol}.
\end{definition}

\begin{definition}
\label{def:autstruct}
A \defterm{right-au\-to\-mat\-ic structure} for $S$ is a pair $(A,L)$, where $A$ is a finite alphabet representing a set of generators for $S$ and $L \subseteq A^+$ is a regular language with $\elt{L} = S$ and such that, for each $a \in A \cup \{\emptyword\}$, 
\[
L_a = \{(u,v) : u,v \in L, \elt{ua} = \elt{v}\}\delta_A
\]
is a regular language over $A(2,\$)$. A \defterm{right-au\-to\-mat\-ic semigroup} is a semigroup that admits a right-au\-to\-mat\-ic structure.

A \defterm{left-au\-to\-mat\-ic} structure for $S$ is a pair $(A,L)$, where $A$ is a finite alphabet representing a set of generators for $S$ and $L \subseteq A^+$ is a regular language with $\elt{L} = S$ and such that, for each $a \in A \cup \{\emptyword\}$, 
\[
{}_aL = \{(u,v) : u,v \in L, \elt{au} = \elt{v}\}\delta_A
\]
is a regular language over $A(2,\$)$. A \defterm{left-au\-to\-mat\-ic semigroup} is a semigroup that admits a left-au\-to\-mat\-ic structure.

By default, `automatic' means `right-au\-to\-mat\-ic'.
\end{definition}

[Hoffmann \& Thomas \cite{hoffmann_notions} also considered the
  possibility of having the padding symbols $\$$ placed on the left rather than
  on the right. However, every cancellative semigroup that is
  right-au\-to\-mat\-ic (respectively, left-au\-to\-mat\-ic) with padding on the
  right is right-au\-to\-mat\-ic (respectively, left-au\-to\-mat\-ic) with padding
  on the left, and vice versa \cite[Remark~8.3]{hoffmann_notions}. As
  this paper is concerned only with group-em\-bed\-dable semigroups, which
  are of course cancellative, it suffices to consider only padding on
  the right.]

\begin{proposition}[{\cite[Proposition~3.5]{campbell_autsg}}]
\label{prop:adjoinid}
If a semigroup $S$ is automatic, then so is $\adjone{S}$, the semigroup formed by adjoining an identity to $S$.
\end{proposition}

\begin{proposition}[{\cite[Theorem~1.1]{duncan_change}}]
\label{prop:changegens}
Let $M$ be a monoid with automatic structure $(A,L)$ and let $B$ represent a finite \textup{[}semigroup\textup{]} generating set for $M$. Then there exists an automatic structure $(B,K)$ for $M$.
\end{proposition}

\begin{theorem}[{\cite[Theorem~3.14]{crr_ssgofg}}]
\label{thm:autsgcharbyfellow}
Let $S$ be a semigroup that embeds in a group. Let $(A,L)$ be a rational structure for $S$. Then $(A,L)$ is an automatic structure for $S$ if and only if there exists a constant $\lambda \in\nset$ such that for all $a \in A \cup \{\emptyword\}$, if $u,v \in L$ are such that $\elt{ua} = \elt{v}$, then for all $t$ the distance from $\elt{u(t)}$ to $\elt{v(t)}$ is less than $\lambda$.
\end{theorem}

\begin{proposition}[{\cite[Proposition~2.3]{campbell_autsg}}]
\label{prop:composition}
Let $U$ and $V$ be subsets of $A^+ \times A^+$ such that $U\delta_A$ and $V\delta_A$ are regular. Then
\[
(U \circ V^{-1})\delta_A = \{(u,v)\delta_A : (\exists w \in A^+)((u,w) \in U \land (v,w) \in V)\}
\]
is also a regular language over $A(2,\$)$.
\end{proposition}

\begin{definition}
Let $A$ be an alphabet. For any word $w = w_1\cdots w_n \in A^+$, where $w_i \in A$, define $w^\rev$ to be $w_n\cdots w_1$.

Let $R \subseteq A^+ \times A^+$, and let $S = \sgpres{A}{\rel{R}}$. Define $\rel{R}^\rev$ by
\[
\rel{R}^\rev = \{(u^\rev,v^\rev) : (u,v) \in \rel{R}\}.
\]
Define $S^\rev$ to be $\sgpres{A}{\rel{R}^\rev}$.
\end{definition}

Notice that $(S^\rev)^\rev = S$.

\begin{lemma}[{\cite[Lemma~3.4]{hoffmann_notions}}]
\label{lem:revaut}
The semigroup $S$ is right-au\-to\-mat\-ic if and only if the semigroup
$S^\rev$ is left-au\-to\-mat\-ic.
\end{lemma}

\section{The Cayley graph of the Baumslag--Solitar semigroups}

\begin{figure}[tb]
\centerline{\includegraphics{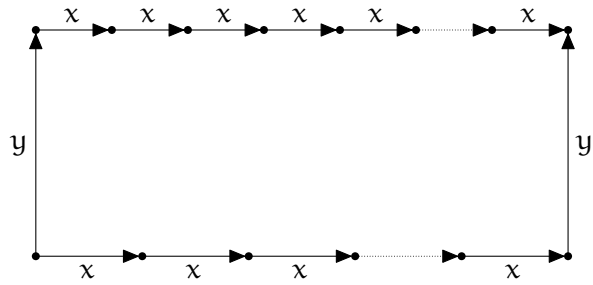}}
\caption{Part of the Cayley graph of the Baumslag--Solitar semigroup $\BSS(m,n)$: a cell describing the relation $(yx^m,x^ny)$.}
\label{fig:bsscg1}
\end{figure}

\begin{figure}[tb]
\centerline{\includegraphics{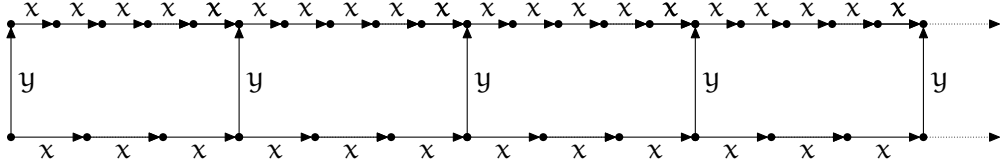}}
\caption{Part of the Cayley graph of the Baumslag--Solitar semigroup $\BSS(5,3)$.}
\label{fig:bsscg2}
\end{figure}

Fix $m,n$ with $m > n$. The present section describes how to construct
the Cayley graph of $\BSS(m,n)$. First of all, consider a single cell
describing the relation $(yx^m,x^ny)$, as shown in
Figure~\ref{fig:bsscg1}. Join copies of this cell along the edges
labelled $y$ as shown in Figure~\ref{fig:bsscg2}; this row of cells is
infinite (to the right).

Let $\omega$ be the basepoint of the graph, from which extends an
infinite horizontal row of edges each labelled by $x$. Now proceed
inductively: to every infinite horizontal row $R$ of edges labelled by
$x$ that has now yet been considered, add $n$ copies of the row of
cells shown in Figure~\ref{fig:bsscg2}, identifying the basepoint of
the $(k-1)$-th such row of cells with the $k$-th vertex from the left
of $R$. Viewed side-on, the fragment of the graph constructed thus far
is a `fan' with $n$ spokes.  [Figures~\ref{fig:bsscg3}
  and~\ref{fig:bsscg3a} show this step in the construction for $m=3$
  and $n=2$.] This procedure inductively constructs the Cayley graph
of $\BSS(m,n)$. Viewed side-on, this graph is an infinite $n$-ary
tree, as shown in Figure~\ref{fig:bsssideon}. Select an infinite path
climbing this tree and take the subset of the Cayley graph that
projects to this path. Call such a subset a \defterm{branch} of the
Cayley graph. [The construction of the Cayley graph just described is
  similar to that for the Baumslag--Solitar group $\BSG(m,n)$ given by
  \cite[Section~7.4]{epstein_wordproc}.]

\begin{figure}[tb]
\centerline{\includegraphics{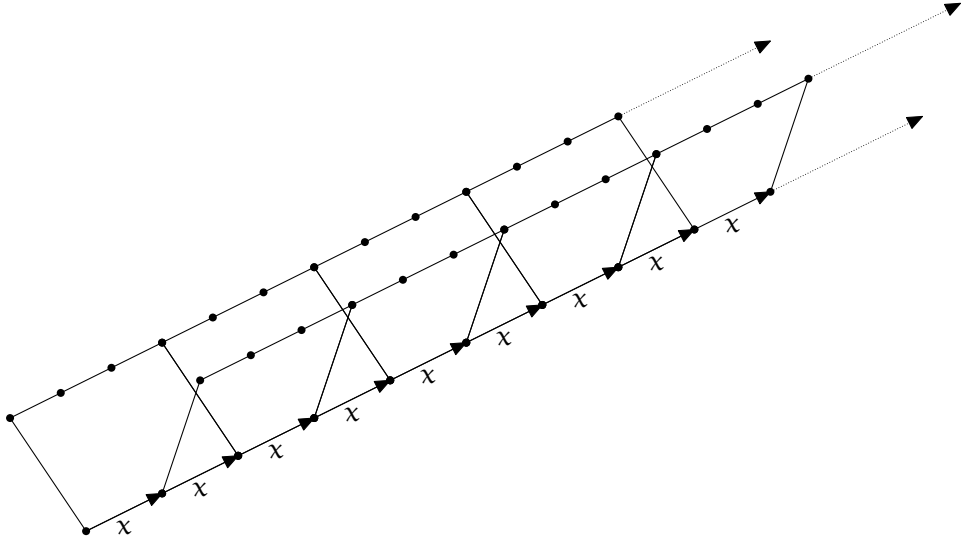}}
\caption{Part of the Cayley graph of the Baumslag--Solitar semigroup $\BSS(3,2)$.}
\label{fig:bsscg3}
\end{figure}

\begin{figure}[tb]
\centerline{\includegraphics{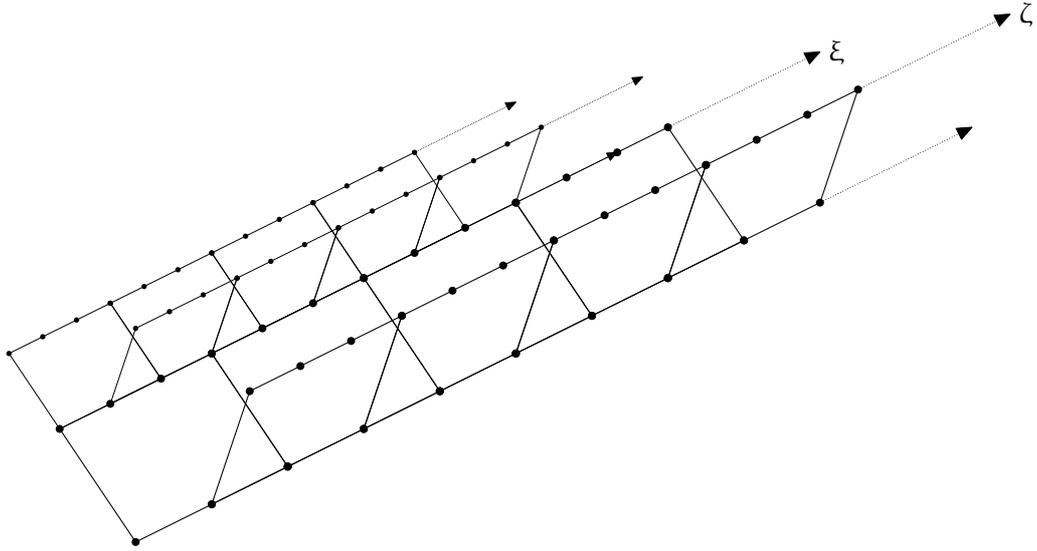}}
\caption{Part of the Cayley graph of the Baumslag--Solitar semigroup $\BSS(3,2)$. Two rows of cells have just been added to the row of $x$-edged marked $\xi$; the next step in the inductive construction is to add two rows of cells to the row of $x$-edges marked $\zeta$.}
\label{fig:bsscg3a}
\end{figure}

\begin{figure}[tb]
\centerline{\includegraphics{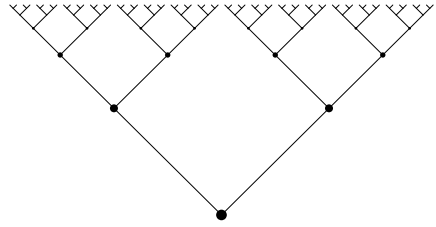}}
\caption{Part of the Cayley graph of the Baumslag--Solitar semigroup $\BSS(3,2)$ viewed `side-on'. The only edges `visible' are labelled by $y$; the edges labelled by $x$ are pointing into the page.}
\label{fig:bsssideon}
\end{figure}

\begin{figure}[tb]
\centerline{\includegraphics{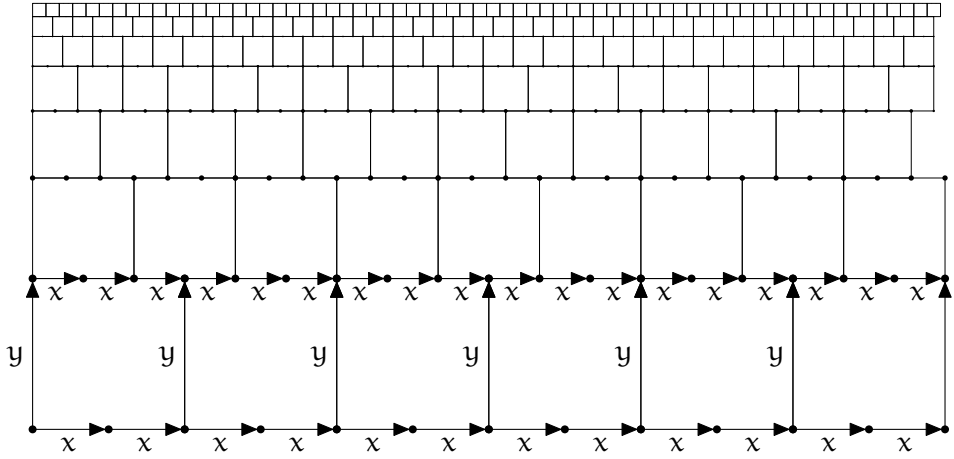}}
\caption{Part of a branch of the Cayley graph of the Baumslag--Solitar semigroup $\BSS(3,2)$ embedded into the Euclidean plane.}
\label{fig:bsscg4}
\end{figure}

A branch of the Cayley graph of $\BSS(m,n)$ may be embedded into the Euclidean plane as shown in Figure~\ref{fig:bsscg4}. Notice that all cells describing a relation $(yx^m,x^ny)$ are similar squares, being scaled by $n/m$ as one climbs from one row to the next. Define a concept of `horizontal distance' within rows of $x$-edges by taking the distance between vertices in the Cayley graph and extending to edges by linear interpolation.

Each element of the Baumslag--Solitar semigroup $\BSS(m,n)$ has a normal form
\[
x^{k_0}yx^{k_1}yx^{k_2}\cdots yx^{k_j}yx^l,
\]
where each $k_i$ is less than $n$; such a normal form can be obtained
from any word over $x$ and $y$ by using the defining relation to move
letters $x$ as far to the right as possible. Identify elements of
$\BSS(m,n)$ with these normal forms.

\section{All subsemigroups right-au\-to\-mat\-ic when $m > n$}

Choose and fix $m,n \in \nset$ with $m>n$. Let $A$ be a finite
alphabet representing a subset of $\BSS(m,n)$. Let $S$ be the
subsemigroup generated by $\elt{A}$. In the Cayley graph
$\Gamma(S,A)$, imagine a word $w \in A^+$ as labelling an edge
from each element $s \in S$ direct to $s\elt{w}$. In an appropriate
branch of the Cayley graph of $\BSS(m,n)$, consider an edge from $s$
to $s\elt{w}$ labelled by $w$. Embed this branch into the Euclidean
plane, so that this edge becomes a straight line between $s$ and
$s\elt{w}$; see Figure~\ref{fig:bsscg5a}. Now, because of the similarity of the various cells
mentioned above, the angle $\theta_w$ between this edge and the
horizontal axis is independent of $s$. Notice further that this angle
must lie between $0$ and $\pi/2$.

Let $A'$ be that subset of $A$ whose letters represent elements of
$\BSS(m,n)$ of the form $x^k$ for some $k \in \nset$ (that is, letters
$a$ such that $\theta_a = 0$). Let $A'' = A - A'$. For each $a \in
A$; let $\beta_a$ be the number of symbols $y$ in $\elt{a}$. (So $a \in A'$ if and only if $\beta_a = 0$.) Similarly for $a \in A$, let $\gamma_a$ be the number of letters $x$ in $\elt{a}$.  Let $g = \max\{\gamma_a : a \in A\}$.

\begin{lemma}
\label{lem:critval}
Let $X$ be a finite subset of $A^+ - (A')^+$. Let $h \in \nset$ be
arbitrary. Then there exists a horizontal distance $\alpha_{h,X} > h$
with the following property: if edges labelled by $w,z \in X$
intersect two adjacent $x$-rows at points $p$, $q$ and $p'$, $q'$,
respectively, and the horizontal distance between $p$ and $q$ exceeds
$\alpha_{h,X}$, then the horizontal distance between $p'$ and $q'$
also exceeds $\alpha_{h,X}$.
\end{lemma}

[This is a purely geometrical lemma. Words $w$ and $z$ may not
  label lines passing through a particular choice of points $p$, $p'$,
  $q$ and $q'$.]

\begin{proof}
Notice that if $w \in X$, then $\elt{w} \neq x^k$ for any $k \in \nset$. So $\theta_w > 0$.

\begin{figure}[tb]
\centerline{\includegraphics{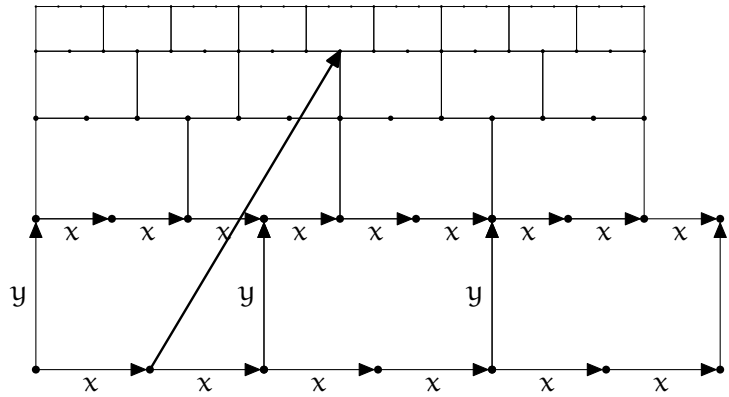}}
\caption{An edge running from $s$ to $s\elt{w}$, where $s = \elt{x}$
  and $w = xyxyy$. Observe that the similarity of the cells in this
  embedding into the Euclidean plane ensures that this edge makes the
  same angle with each row of $x$-edges, and that this angle is
  independent of $s$.}
\label{fig:bsscg5a}
\end{figure}

\begin{figure}[tb]
\centerline{\includegraphics{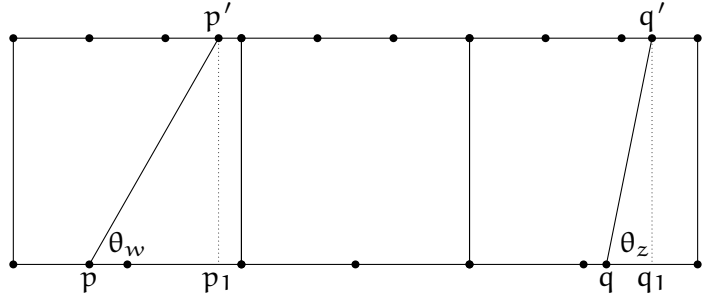}}
\caption{An example of the variance of horizontal distance from one
  $x$-row to the next. Notice that the horizontal distance from $p'$
  to $q'$ is $m/n$ (in this case $3/2$) times that between $p_1$ and
  $q_1$.}
\label{fig:bsscg5}
\end{figure}

Let $d$ be the horizontal distance between $p$ and $q$. Then the horizontal distance between $p'$ and $q'$ is given by:
\begin{equation}
\label{eq:hordist}
\frac{m}{n}\big|d - \cotan \theta_w + \cotan\theta_z\big|,
\end{equation}
as can be seen from Figure~\ref{fig:bsscg5}. [The absolute value is
  needed in case the two lines cross over. The `vertical' distance
  between the two $x$-rows is $1$. Although all concepts of distance
  discussed here relate to the Cayley graph, one must momentarily
  appeal to the Euclidean plane for trigonometric purposes.] Now,
$m>n$ so that $m/n >1$.


Let
\[
\alpha_{h,X} = \max\big[\{h\} \cup \big\{\cotan \theta_w - \cotan\theta_z : w,z\in X\big\}\big]\frac{m}{m-n};
\]
since $X$ is finite the maximum exists. If $d > \alpha_{h,X}$, and
$w,z \in X$, then
\begin{align*}
d &> \big(\cotan\theta_w - \cotan\theta_z\big)\frac{m}{m-n}\\
d(1 - n/m) &> \cotan\theta_w - \cotan\theta_z\\
\frac{m}{n}\big(d - \cotan\theta_w + \cotan\theta_z\big) &> d.
\end{align*}
That is, if the horizontal distance between $p$ and $q$ exceeds $\alpha_{h,X}$, then the distance between $p'$ and $q'$ is larger still. This completes the proof.
\end{proof}

As a consequence of the defining relation $(yx^m, x^ny)$, for any $a
\in A'$,
\[
y^k \elt{a^{m^k}} = \elt{a^{n^k}}y^k,
\]
and, more generally, for $b \in A''$ and $k = \beta_b$,
\[
\elt{b a^{m^k}} = \elt{a^{n^k}b}.
\]
Moreover, $\elt{a_1a_2} = \elt{a_2a_1}$, where $a_1,a_2 \in A'$. The
upshot of this is that every element of $S$ has a representative in
the set
\begin{equation}
\label{eq:ssgofbssnf1}
(A')^*K^* - \{\emptyword\},
\end{equation}
where $K$ is the finite set
\[
\big\{ba_1^{\alpha_1}\cdots a_l^{\alpha_l} : b \in A'', 0 \leq \alpha_i < m^{\beta_b}\big\}
\]
and $A' = \{a_1,\ldots,a_l\}$. [In a way, this is the reverse of the
  set of normal forms for $\BSS(m,n)$ --- letters representing powers
  of $x$ are now moved as far \emph{left} as possible.]

Let $1$ be a new symbol representing the adjoined identity of
$\adjone{S}$. Let $k$ be the maximum length of any element of
$K$. Define
\[
K' = \big\{bw1^{\beta_b (k+1)-|w|-1} : b \in A'', bw \in K\}.
\]
The set $K'$ consists of elements $bw$ of $K$ padded with symbols $1$
to a length that is a constant multiple of $\beta_b$. [The large
  constant multiple $(k+1)$ is necessary to ensure that the string of
  symbols $1$ always has positive length.] Define
\[
J = \{a1^{\gamma_a - 1} : a \in A'\}.
\]
(Recall that $\elt{a} = x^{\gamma_a}$ for each $a \in A'$.) 

Let $L = J^*(K')^* - \{\emptyword\} \cup \{1\}$. The language $L$
differs from the set \eqref{eq:ssgofbssnf1} only by padding using
symbols $1$ and the addition of the word $1$ to represent the adjoined
identity. Therefore, since the set \eqref{eq:ssgofbssnf1} maps onto
$S$, the language $L$ maps onto $\adjone{S}$. The aim is now to show
that $(A \cup \{1\},L)$ is an automatic structure for $\adjone{S}$.

Suppose $u,v \in L$ and $a \in A \cup \{\emptyword\}$ with $\elt{ua} =
\elt{v}$. Let $u = u'u''$ and $v = v'v''$, where $u',v' \in J^*$ and
$u'',v'' \in (K')^*$. The paths $\path{ua}$ and $\path{v}$ run from
the basepoint $\omega$ to a common vertex. As these paths never run `downwards'
through the Cayley graph, they lie in a common branch containing
$\elt{ua} = \elt{v}$. Isolate such a branch and embed it into the
Euclidean plane. The parts of the two paths labelled by $u'$ and $v'$
run along the lowest $x$-row.

Let $t \in \nset \cup \{0\}$. Suppose firstly that $t \leq
\min\{|u'|,|v'|\}$. As any word $w$ in $J$ has length equal to the
number of letters $x$ in $\elt{w}$, the same holds true for any word
in $J^*$. Any prefix of a word in $J^*$ is at most $M = \max\{|w| : w
\in J\}$ letters short of a member of $J^*$; the number of letters $x$
in $\elt{u'(t)}$ differs by at most $M$ from $t$. Similar reasoning
applies to $v'$; the distance between $\elt{u'(t)}$ and $\elt{v'(t)}$
is therefore bounded.

The words $u''$ and $v''$ label subpaths from $\elt{u'}$ and
$\elt{v'}$ to $\elt{u}$ and $\elt{v}$.  It is clear that the
horizontal distance from $\elt{u}$ to the intersection of path
labelled by $v''$ with the $x$-row in which $\elt{u}$ lies is at most
$g = \max\{\gamma_a : a\in A\}$.  Imagine these paths as made up of
`segments' $w \in K'$, with each $w$ labelling an edge that runs
directly from $s$ to $s\elt{w}$. Consider the intersection of the
paths with a given $x$-row, at points $p$ and $q$. Let the
intersections with the next $x$-row be $p'$ and $q'$. Let $w$ and $z$
be the labels on the segments that run between $p$ and $p'$ and $q$
and $q'$. (These segments may of course start \emph{below} $p$ and $q$
and end \emph{above} $p'$ and $q'$.) By Lemma~\ref{lem:critval} above,
if the horizontal distance between $p$ and $q$ exceeds
$\alpha_{g,K'}$, then the distance between $p'$ and $q'$ also exceeds
$\alpha_{g,K'}$, regardless of $w$ and $z$. (Recall that $g$ is the
maximum number of letters of $x$ in $\elt{a}$ for $a \in A$.)
Therefore, since the paths labelled by $u''$ and $v''$ must eventually
be a horizontal distance at most $g$ apart, the horizontal distance
between their intersections with each $x$-row cannot exceed this
critical value $\alpha_{g,X}$.

In particular, the points $\elt{u'}$ and $\elt{v'}$ can only be a
bounded distance apart. Therefore $|u'|$ and $|v'|$ can only differ by
a bounded amount.

An argument similar to that for $J^*$ shows that if $w$ is a prefix of
a word in $(K')^*$, then the length of $w$ differs from a constant
multiple of the number of letters $y$ in $\elt{w}$ by only a bounded
amount.

Suppose now that $t \geq \min\{|u'|,|v'|\}$. Consider the elements
$\elt{u(t)}$ and $\elt{v(t)}$. By altering $t$ by a bounded amount,
assume $v(t) \in v'(K')^*$. By the observation in the last paragraph,
the [new] elements $\elt{u(t)}$ and $\elt{v(t)}$ lie on $x$-rows that
are only a bounded number of elements $y$ apart. Therefore
$\elt{u(t)}$ is a bounded distance from the intersection $p$ of the
subpath labelled by $u''$ with the $x$-row containing
$\elt{v(t)}$. The horizontal distance between $p$ and $\elt{v(t)}$
cannot exceed $\alpha_{h,X}$. Therefore the distance between
$\elt{u(t)}$ and $\elt{v(t)}$ in the Cayley graph of $\BSS(m,n)$ is
bounded. Restoring the original value for $t$ does not alter this
fact. 

By Theorem~\ref{thm:autsgcharbyfellow}, $(A \cup \{1\},L)$ is an
automatic structure for $\adjone{S}$. The semigroup $S$ is therefore
automatic by Proposition~\ref{prop:adjoinid}. Since $S$ was an
arbitrary finitely generated subsemigroup of $\BSS(m,n)$, this proves
the following result.

\begin{theorem}
\label{thm:one}
Every finitely generated subsemigroup of the Baumslag--Solitar
semigroup $\BSS(m,n)$, where $m > n$, is automatic.
\end{theorem}

\section{All subsemigroups left-au\-to\-mat\-ic when $m < n$}

Theorem~\ref{thm:one} has the following left-handed version:

\begin{theorem}
\label{thm:two}
Every finitely generated subsemigroup of the Baumslag--Solitar
semigroup $\BSS(m,n)$, where $m < n$, is left-au\-to\-mat\-ic.
\end{theorem}

\begin{proof}
Note first of all that
\[
\BSS(m,n)^\rev = \sgpres{x,y}{(yx^m,x^ny)}^\rev = \sgpres{x,y}{(x^my,yx^n)} = \BSS(n,m).
\]
Let $S$ be a finitely generated subsemigroup of $\BSS(m,n)$. Then
$S^\rev$ is a subsemigroup of $\BSS(n,m)$ and is therefore
right-au\-to\-mat\-ic by Theorem~\ref{thm:one}. So $S$ itself is
left-au\-to\-mat\-ic by Lemma~\ref{lem:revaut}.
\end{proof}

\section{Non-au\-to\-mat\-ic subsemigroups when $m=n$}
\label{sec:bssmmcontnonaut}

This section shows that the Baumslag--Solitar semigroup $\BSS(m,m)$,
where $m \geq 2$ contains finitely generated subsemigroups that are
neither right- not left-au\-to\-mat\-ic. First of all,
Proposition~\ref{prop:bssmmcontfxn} shows that $\BSS(m,m)$ contains
the direct product of a free semigroup of rank $m$ and the natural
numbers. Example~\ref{ex:ssgofdpfsnnaut} then explicitly exhibits a
finitely generated non-au\-to\-mat\-ic subsemigroup of this direct product.

\begin{proposition}
\label{prop:bssmmcontfxn}
The Baumslag--Solitar semigroup $\BSS(m,m)$ contains the semigroup
$\{p_1,\ldots,p_m\}^* \times (\nset \cup \{0\}) - \{(\emptyword,
0)\}$: the free product of the free monoid on $m$ letters and the
natural numbers {\rm(}including zero\/{\rm)} with the identity
$(\emptyword, 0)$ removed.
\end{proposition}

\begin{proof}
Let $A = \{p_1,\ldots,p_m,r\}$ be an alphabet representing elements of
$\BSS(m,m)$ as follows
\[
\elt{p_i} = x^{i-1}y\text{ for each $i$, and }\elt{r} = x^m.
\]
Let $S$ be the subsemigroup of $\BSS(m,m)$ generated by $\elt{A}$. The
aim is to show that $S$ is presented by
\[
\sgpres{A}{(p_ir,rp_i) \text{ for all $i$}}.
\]

To prove this, note firstly that every relation $(p_ir,rp_i)$ holds in
$S$. Define a set of normal forms $N = \{p_1,\ldots,p_m\}^*r^* -
\{\emptyword\}$. Every element of $S$ has such a normal form: letters $r$
can be moved to the left of all letters $p_i$ using the defining
relations. Consider any element of $s \in S$. Suppose first that $s$
contains some letter $y$. If $s$ begins $x^{i-1}y\cdots$, then any
word in $N$ representing it must begin $p_i\cdots$. On the other hand,
if $s$ contains no letters $y$, then $s = x^{m\alpha}$ for some
$\alpha \in \nset$ and the normal form word representing it is
$r^\alpha$. In the first case, one can cancel the $x^{i-1}y$ and
iterate this reasoning to obtain the entire normal form word
representing $s$. Thus $N$ is a set of unique normal forms for $S$,
and so $S$ has the given presentation.

Thus $S$ is isomorphic to $\{p_1,\ldots,p_m\}^* \times (\nset \cup
\{0\}) - \{(\emptyword, 0)\}$.
\end{proof}

\begin{example}
\label{ex:ssgofdpfsnnaut}
Let $A = \{a,b,c,d,e,f,g,h\}$ be an alphabet representing elements of
$\{x,y,p,q,r\}^* \times (\nset\cup\{0\})$ as follows:
\begin{align*}
\elt{a} &= (x^2p,0),\\
\elt{b} &= (qrp,1), & \elt{f} &= (x^2pq,0),\\
\elt{c} &= (qr,0), & \elt{g} &= (rpqrpq,3),\\
\elt{d} &= (pqr,2), & \elt{h} &= (rpy^2,0).\\
\elt{e} &= (py^2,0),
\end{align*}
Let $S$ be the semigroup generated by $\elt{A}$. 
\end{example}

Notice that, since any free semigroup of finite rank embeds into the
free semigroup of rank $2$, which in turn embeds into any free
semigroup of rank $m \geq 2$, the product $\{x,y,p,q,r\}^* \times
(\nset\cup\{0\})$ embeds into $\{p_1,\ldots,p_m\}^* \times
(\nset\cup\{0\})$. So this semigroup $S$ embeds into any
Baumslag--Solitar semigroup $\BSS(m,m)$ with $m \geq 2$ by
Proposition~\ref{prop:bssmmcontfxn}.

\begin{proposition}
The semigroup $S$ is neither right-au\-to\-mat\-ic nor left-au\-to\-mat\-ic.
\end{proposition}

\begin{proof}\relax
First, it will be shown that $S$ is not right-au\-to\-mat\-ic. [The proof is
  similar in spirit to that of \cite[Proposition~3.1]{c_nonaut}.] The
proof that $S$ is not left-au\-to\-mat\-ic is essentially parallel; this
will be discussed briefly afterwards.

First of all, notice that for each $\alpha \in \nset$,
\[
\elt{ab^\alpha cd^\alpha e} = (x^2(pqr)^{2\alpha+1}py^2,3\alpha) = \elt{fg^\alpha h}.
\]
Elementary reasoning shows that for each $\alpha \in \nset \cup \{0\}$
the elements
\[
\elt{ab^\alpha cd^\alpha} = (x^2(pqr)^{2\alpha+1},3\alpha), \qquad \elt{fg^\alpha} = (x^2(pqr)^{2\alpha}pq,3\alpha)
\]
have unique representives $ab^\alpha cd^\alpha$ and $fg^\alpha$ over
the alphabet $A$. [The more complicated case is
  $(x^2(pqr)^{2\alpha+1},3\alpha)$: considering the free semigroup
  component shows that any word representing this element must be of
  the form $ab^\beta cd^\gamma$ with $\beta + \gamma = 2\alpha$. The
  $\nset$-component requires that $\beta + 2\gamma =
  3\alpha$. Together this forces $\beta = \gamma = \alpha$.]

Suppose, with the aim of obtaining a contradiction, that $S$ is
automatic. Then, by Proposition~\ref{prop:adjoinid}, so is
$S^1$. Proposition~\ref{prop:changegens} implies that $S^1$ has an
automatic structure $(C,L)$, where $C = A \cup \{1\}$. (The new symbol
$1$ represents the adjoined identity of $S^1$.) Let $\phi : C^* \to
A^*$ map $w \in C^*$ to the word over $A$ formed by deleting any
symbols $1$ from $w$. Obviously $\elt{w\phi} = \elt{w}$.

Proposition~\ref{prop:composition} shows that the language
\begin{align*}
L_e \circ L_h^{-1} &= \{(u,w)\delta_C : u,w \in L, \elt{ue} = \elt{w}\} \circ \{(w,v)\delta_C : w,v \in L, \elt{vh} = \elt{w}\}\\
&= \{(u,v)\delta_C : u,v \in L, \elt{ue} = \elt{vh}\}
\end{align*}
is regular. Let $N$ be the number of states in a finite state automaton $\fsa{A}$ recognizing $L_e \circ L_h^{-1}$.

For each $\alpha \in \nset \cup \{0\}$, let $u_\alpha$ and $v_\alpha$
be representatives in $L$ of the elements $\elt{ab^\alpha cd^\alpha}$
and $\elt{fg^\alpha}$, respectively. Since $\elt{ab^\alpha cd^\alpha}$
has a unique representative over $A$, it is clear that $u_\alpha\phi =
ab^\alpha cd^\alpha$. Similarly, $v_\alpha\phi = fg^\alpha$. (So
$u_\alpha$ and $v_\alpha$ are the words $ab^\alpha cd^\alpha$ and
$fg^\alpha$ with some symbols $1$ possibly inserted.) By its
definition, the language $L_e \circ L_h^{-1}$ contains
$(u_\alpha,v_\alpha)\delta_C$ for all $\alpha \in \nset \cup \{0\}$.

Fix $\alpha > N$. Consider the automaton $\fsa{A}$ reading
$(u_\alpha,v_\alpha)\delta_C$, and the states it enters immediately
after reading each of the letters $b$ from the word $u_\alpha$. Since
the number of letters $b$ exceeds $N$, the automaton enters the same
state after reading two different letters $b$. Let $u'$ and $u'u''$ be
the prefixes of $u_\alpha$ up to and including these two different
letters $b$. That is, $u'\phi = ab^\beta$, $(u'u'')\phi = ab^\gamma$, for some $\beta,\gamma \in \nset$ with $\gamma > \beta$. Let $v'$ and $v'v''$ be
prefixes of $v_\alpha$ of the same lengths as $u'$ and $u'u''$,
respectively. The subword $v''$ is such that $v''\phi = fg^\eta$ or
$v''\phi = g^\eta$ for some $\eta \in \nset \cup \{0\}$. (The former
possibility arises because $v'$ may be a string of symbols~$1$.) So pumping $(u'',v'')\delta_C$ shows that there is $(w,z)\delta_C \in L_e \circ L_h^{-1}$ with either
\[
w\phi = ab^\beta b^{2(\gamma - \beta)}b^{\alpha-\gamma}cd^\alpha \text{ and } z\phi = fg^\eta fg^\alpha,
\]
or
\[
w\phi = ab^\beta b^{2(\gamma - \beta)}b^{\alpha-\gamma}cd^\alpha \text{ and } z\phi = fg^{\alpha+\eta}.
\]
In both cases, by the definition of $L_e \circ L_h^{-1}$, it follows that $\elt{we} = \elt{zh}$. In the former case, this implies that
\[
(x^2(pqr)^{2\alpha+\gamma-\beta+1}py^2,3\alpha+\gamma-\beta) = (x^2(pqr)^{2\eta}pq x^2(pqr)^{2\alpha + 1}py^2,3\alpha+3\eta);
\]
in the latter, that
\[
(x^2(pqr)^{2\alpha+\gamma-\beta+1}py^2,3\alpha+\gamma-\beta) = (x^2(pqr)^{2\alpha+2\eta+1}py^2,3\alpha+3\eta).
\]
In the former case, the free semigroup components do not match, which
is a contradiction. In the latter, for the free semigroup components
to match, $2\eta = \gamma - \beta$. But for the $\nset$-components to
match, $3\eta =\gamma-\beta$, which forces $\gamma-\beta = 0$,
contradicting $\gamma > \beta$. So both cases lead to a contradiction:
hence $S$ cannot be automatic.

To see that $S$ is not left-au\-to\-mat\-ic, proceed in the same way, but
use elements $\elt{b^\alpha cd^\alpha e}$ and $\elt{g^\alpha h}$,
and the regular language ${}_aL \circ {}_fL^{-1}$.
\end{proof}

\section{Appendix: Malcev presentations}
\label{sec:malcev}

Malcev presentations are a special type of semigroup presentation for
semigroups embeddable into groups. Informally, a Malcev presentation
defines a semigroup by means of generators, defining relations, and a
rule of group-embedability. This rule of group-embeddability is worth
an infinite number of defining relations, in the sense that a
semigroup can admit a finite Malcev presentation but no finite
ordinary presentation. The present section discusses what implications
Theorems~\ref{thm:one} and~\ref{thm:two} and
Example~\ref{ex:ssgofdpfsnnaut} have, in the light of
Theorem~\ref{thm:autagfinmalpres} below, for the theory of Malcev
presentations.

Malcev presentations were introduced by Spehner in 1977
\cite{spehner_presentations}, though they are based on Malcev's
necessary and sufficient condition for the embeddability of a
semigroup in a group \cite{malcev_immersion1} (see
also~\cite[Chapter~12]{clifford_semigroups2}). The theory of Malcev
presentations was relatively inactive until recent work by the present
author and collaborators
\cite{c_phdthesis,c_ssgofdp,crr_ssgofg,crr_vfree}; see also the survey
article~\cite{c_malcev}.

\begin{definition}
Let $S$ be any semigroup. A congruence $\sigma$ on $S$ is a
\defterm{Malcev congruence} if $S/\sigma$ is embeddable in a group.
\end{definition}

If $\{\sigma_i : i \in I\}$ is a set of Malcev congruences on $S$,
then $\sigma = \bigcap_{i \in I} \sigma_i$ is also a Malcev congruence
on $S$. This is true because $S/\sigma_i$ embeds in a group $G_i$ for
each $i \in I$, so $S/\sigma$ embeds in $\prod_{i \in I} S/\sigma_i$,
which in turn embeds in $\prod_{i \in I} G_i$. The following
definition therefore makes sense.

\begin{definition}
\label{def:malpres}
Let $A^+$ be a free semigroup; let $\rho \subseteq A^+ \times A^+$ be
any binary relation on $A^+$. Denote by $\malcgen{\rho}$ the smallest
Malcev congruence containing $\rho$ --- namely,
\[
\malcgen{\rho} = \bigcap \left\{\sigma : \sigma \supseteq \rho, \text{ $\sigma$ is a Malcev congruence on }A^+\right\}.
\]
Then $\sgmpres{A}{\rho}$ is a \defterm{Malcev presentation} for [any
  semigroup isomorphic to] $A^+\!/\malcgen{\rho}$.  If both $A$ and
$\rho$ are finite, the the Malcev presentation $\sgmpres{A}{\rho}$ is
said to be finite.

A group-em\-bed\-dable semigroup is \defterm{Malcev coherent} if all of
its finitely generated subsemigroups admit finite Malcev
presentations.
\end{definition}

Several classes of Malcev coherent semigroups are known: virtually
free groups \cite[Theorem~3]{crr_vfree}, virtually nilpotent groups
\cite[Theorem~1]{crr_ssgofg}, direct products of virtually free and
abelian groups \cite[Theorem~2]{c_ssgofdp}, and free products of free
monoids and abelian groups \cite[Theorem~6]{crr_ssgofg}. [Each of
  these classes contain finitely generated semigroups that do not
  admit finite \emph{ordinary} presentations; for this reason, their
  Malcev coherence is of interest. See \cite[Table~3]{c_malcev} for a
  list of semigroups known to be Malcev coherent or not Malcev
  coherent.]

\begin{theorem}[{\cite[Theorem~2]{crr_ssgofg}}]
\label{thm:autagfinmalpres}
Every right- or left-au\-to\-mat\-ic group-em\-bed\-dable semigroup admits a
finite Malcev presentation.
\end{theorem}

(Although \cite{crr_ssgofg} only contains the proof for
right-au\-to\-mat\-ic semigroups, the proof for left-au\-to\-mat\-ic semigroups is
almost identical.)

Since every finitely generated subsemigroup of a
Baumslag--Solitar semigroup $\BSS(m,n)$ with $m \neq n$ is either
right- or left-au\-to\-mat\-ic (Theorems~\ref{thm:one} and~\ref{thm:two}), and since every right- or left-au\-to\-mat\-ic group-em\-bed\-dable semigroup
admits a finite Malcev presentation
(Theorem~\ref{thm:autagfinmalpres}), the following
result obtains:

\begin{corollary}
\label{corol:bsmalcoh}
The Baumslag--Solitar semigroup $\BSS(m,n)$, where $m \neq n$, is
Malcev coherent.
\end{corollary}

This raises the following question:

\begin{problem}
\label{prob:bssmcoh}
Are the Baumslag--Solitar semigroups $\BSS(m,m)$ (where $m \geq 2$)
Malcev coherent?
\end{problem}

[Notice that $\BSS(1,1)$ is abelian and so Malcev coherent by
  R\'{e}dei's theorem \cite{redei_commutative}.]  The author
conjectures that Open Problem~\ref{prob:bssmcoh} has a positive
answer. However, the conclusions of Section~\ref{sec:bssmmcontnonaut}
rule out a proof using automatic structures.

The following question naturally arises from the coherence of
Baumslag--Solitar groups \cite{ar:kropholler1990a}:

\begin{problem}
\label{prob:bsgmcoh}
Are Baumslag--Solitar groups Malcev coherent?
\end{problem}

Baumslag \cite[Section~B]{baumslag_problems} asks whether all
one-relator groups are coherent. Some progress has been made on this
front \cite{ar:karrass1970a,ar:mccammond2005a}. It is therefore
natural, although perhaps precipitate, to pose the following question:

\begin{problem}
\label{prob:onerelmcoh}
Are all one-relator groups Malcev coherent?
\end{problem}

A restricted version of this question that may be easier to answer is
the following:

\begin{problem}
\label{prob:onerelcancsgmcoh}
Are all one-relation cancellative semigroups Malcev coherent? (Adyan's
Theorem \cite{adjan_embeddability} ensures that one-relation
cancellative semigroups are group-em\-bed\-dable.)
\end{problem}


\begin{thebibliography}{ECH{\etalchar{+}}92}

\bibitem[Ady60]{adjan_embeddability}
S.~I. Adyan.
\newblock `On the embeddability of semigroups in groups'.
\newblock {\em Soviet Math. Dokl.}, 1 (1960), pp. 819--821.
\newblock [Translated from the Russian.].

\bibitem[Bau74]{baumslag_problems}
G. Baumslag.
\newblock `Some problems on one-relator groups'.
\newblock In {\em Proceedings of the {S}econd {I}nternational {C}onference on               the {T}heory of {G}roups ({A}ustralian {N}at. {U}niv.,               {C}anberra, 1973)}, pp. 75--81. Lecture Notes in Math., Vol. 372, Berlin, 1974. Springer.
\newblock {\sc doi:} \href {http://dx.doi.org/10.1007/BFb0065160} {{10.1007/BFb0065160}}.

\bibitem[BS62]{ar:baumslag1962a}
G. Baumslag \& D. Solitar.
\newblock `Some two-generator one-relator non-{H}opfian groups'.
\newblock {\em Bull. Amer. Math. Soc.}, 68 (1962), pp. 199--201.
\newblock {\sc doi:} \href {http://dx.doi.org/10.1090/S0002-9904-1962-10745-9} {{10.1090/S0002-9904-1962-10745-9}}.

\bibitem[Cai05]{c_phdthesis}
A.~J. Cain.
\newblock {\em Presentations for Subsemigroups of Groups}.
\newblock {Ph.D.\ Thesis}, University of St~Andrews, 2005.
\newblock {\sc url:} \href{http://www-groups.mcs.st-andrews.ac.uk/~alanc/pub/c_phdthesis.pdf}{\nolinkurl{www-groups.mcs.st-andrews.ac.uk/~alanc/pub/c_phdthesis.pdf}}.

\bibitem[Cai06]{c_nonaut}
A.~J. Cain.
\newblock `A group-embeddable non-automatic semigroup whose universal               group is automatic'.
\newblock {\em Glasg. Math. J.}, 48, no.~2 (2006), pp. 337--342.
\newblock {\sc doi:} \href {http://dx.doi.org/10.1017/s0017089506003107} {{10.1017/s0017089506003107}}.

\bibitem[Cai07]{c_malcev}
A.~J. Cain.
\newblock `Malcev presentations for subsemigroups of groups --- a survey'.
\newblock In C.~M. Campbell, M. Quick, E.~F. Robertson, \& G.~C. Smith, eds, {\em Groups St Andrews 2005 {\rm (Vol. 1)}}, no. 339 in {\em London Mathematical Society Lecture Note Series}, pp. 256--268, Cambridge, 2007. Cambridge University Press.

\bibitem[Cai09]{c_ssgofdp}
A.~J. Cain.
\newblock `Malcev presentations for subsemigroups of direct products of coherent groups'.
\newblock {\em J. Pure Appl. Algebra}, 213, no.~6 (2009), pp. 977--990.
\newblock {\sc doi:} \href {http://dx.doi.org/10.1016/j.jpaa.2008.10.006} {{10.1016/j.jpaa.2008.10.006}}.

\bibitem[CP67]{clifford_semigroups2}
A.~H. Clifford \& G.~B. Preston.
\newblock {\em {The Algebraic Theory of Semigroups {\rm (Vol.~II)}}}.
\newblock No.~7 in {\em Mathematical Surveys}. American Mathematical Society, Providence, R.I., 1967.

\bibitem[CRR06a]{crr_ssgofg}
A.~J. Cain, E.~F. Robertson, \& N. Ru{\v{s}}kuc.
\newblock `Subsemigroups of groups: presentations, {M}alcev               presentations, and automatic structures'.
\newblock {\em J. Group Theory}, 9, no.~3 (2006), pp. 397--426.
\newblock {\sc doi:} \href {http://dx.doi.org/10.1515/jgt.2006.027} {{10.1515/jgt.2006.027}}.

\bibitem[CRR06b]{crr_vfree}
A.~J. Cain, E.~F. Robertson, \& N. Ru{\v{s}}kuc.
\newblock `Subsemigroups of virtually free groups: finite {M}alcev               presentations and testing for freeness'.
\newblock {\em Math. Proc. Cambridge Philos. Soc.}, 141, no.~1 (2006), pp. 57--66.
\newblock {\sc doi:} \href {http://dx.doi.org/10.1017/s0305004106009236} {{10.1017/s0305004106009236}}.

\bibitem[CRRT01]{campbell_autsg}
C.~M. Campbell, E.~F. Robertson, N. Ru{\v{s}}kuc, \& R.~M. Thomas.
\newblock `Automatic semigroups'.
\newblock {\em Theoret. Comput. Sci.}, 250, no.~1--2 (2001), pp. 365--391.
\newblock {\sc doi:} \href {http://dx.doi.org/10.1016/S0304-3975(99)00151-6} {{10.1016/S0304-3975(99)00151-6}}.

\bibitem[DRR99]{duncan_change}
A.~J. Duncan, E.~F. Robertson, \& N. Ru{\v{s}}kuc.
\newblock `Automatic monoids and change of generators'.
\newblock {\em Math. Proc. Cambridge Philos. Soc.}, 127, no.~3 (1999), pp. 403--409.
\newblock {\sc doi:} \href {http://dx.doi.org/10.1017/S0305004199003722} {{10.1017/S0305004199003722}}.

\bibitem[ECH{\etalchar{+}}92]{epstein_wordproc}
D.~B. A. Epstein, J.~W. Cannon, D.~F. Holt, S.~V. F. Levy, M.~S. Paterson, \& W.~P. Thurston.
\newblock {\em {Word Processing in Groups}}.
\newblock Jones \& Bartlett, Boston, Mass., 1992.

\bibitem[Hof01]{hoffmann_phd}
M. Hoffmann.
\newblock {\em Automatic Semigroups}.
\newblock {Ph.D.\ Thesis}, University of Leicester, 2001.

\bibitem[HT03]{hoffmann_notions}
M. Hoffmann \& R.~M. Thomas.
\newblock `Notions of automaticity in semigroups'.
\newblock {\em Semigroup Forum}, 66, no.~3 (2003), pp. 337--367.
\newblock {\sc doi:} \href {http://dx.doi.org/10.1007/s002330010161} {{10.1007/s002330010161}}.

\bibitem[HT06]{hoffmann_geometric}
M. Hoffmann \& R.~M. Thomas.
\newblock `A geometric characterization of automatic semigroups'.
\newblock {\em Theoret. Comput. Sci.}, 369, no.~1-3 (2006), pp. 300--313.
\newblock {\sc doi:} \href {http://dx.doi.org/10.1016/j.tcs.2006.09.008} {{10.1016/j.tcs.2006.09.008}}.

\bibitem[Jac02]{ar:jackson2002a}
D.~A. Jackson.
\newblock `Decision and separability problems for {B}aumslag--{S}olitar semigroups'.
\newblock {\em Internat. J. Algebra Comput.}, 12, no.~1--2 (2002), pp. 33--49.
\newblock {\sc doi:} \href {http://dx.doi.org/10.1142/S0218196702000857} {{10.1142/S0218196702000857}}.

\bibitem[Kro90]{ar:kropholler1990a}
P.~H. Kropholler.
\newblock `Baumslag--{S}olitar groups and some other groups of cohomological dimension two'.
\newblock {\em Comment. Math. Helv.}, 65, no.~4 (1990), pp. 547--558.
\newblock {\sc doi:} \href {http://dx.doi.org/10.1007/BF02566625} {{10.1007/BF02566625}}.

\bibitem[KS70]{ar:karrass1970a}
A. Karrass \& D. Solitar.
\newblock `The subgroups of a free product of two groups with an               amalgamated subgroup'.
\newblock {\em Trans. Amer. Math. Soc.}, 150 (1970), pp. 227--255.
\newblock {\sc doi:} \href {http://dx.doi.org/10.2307/1995492} {{10.2307/1995492}}.

\bibitem[Mal39]{malcev_immersion1}
A.~I. Malcev.
\newblock `On the immersion of associative systems in groups'.
\newblock {\em Mat.\ Sbornik}, 6, no.~48 (1939), pp. 331--336.
\newblock [In Russian.].

\bibitem[MW05]{ar:mccammond2005a}
J.~P. McCammond \& D.~T. Wise.
\newblock `Coherence, local quasiconvexity, and the perimeter of               2-complexes'.
\newblock {\em Geom. Funct. Anal.}, 15, no.~4 (2005), pp. 859--927.
\newblock {\sc doi:} \href {http://dx.doi.org/10.1007/s00039-005-0525-8} {{10.1007/s00039-005-0525-8}}.

\bibitem[R{\'{e}}d65]{redei_commutative}
L. R{\'{e}}dei.
\newblock {\em The Theory of Finitely Generated Commutative Semigroups}.
\newblock Pergamon Press, Oxford, 1965.
\newblock [Translated from the German. Edited by N.~Reilly.].

\bibitem[Spe77]{spehner_presentations}
J.~C. Spehner.
\newblock `Pr\'esentations et pr\'esentations simplifiables d'un mono{\"{\i}}de simplifiable'.
\newblock {\em Semigroup Forum}, 14, no.~4 (1977), pp. 295--329.
\newblock [In French.].
\newblock {\sc doi:} \href {http://dx.doi.org/10.1007/BF02194675} {{10.1007/BF02194675}}.

\bibitem[SS04]{silva_geometric}
P.~V. Silva \& B. Steinberg.
\newblock `A geometric characterization of automatic monoids'.
\newblock {\em Q. J. Math.}, 55, no.~3 (2004), pp. 333--356.
\newblock {\sc doi:} \href {http://dx.doi.org/10.1093/qjmath/55.3.333} {{10.1093/qjmath/55.3.333}}.

\end{thebibliography}

\newcommand{\etalchar}[1]{$^{#1}$}

\end{document}